\documentclass[11pt]{article}
\usepackage{mathrsfs} 
\usepackage[margin=1.25in]{geometry}
\usepackage{ifthen}
\usepackage{graphicx}
\usepackage{epsfig}
\usepackage{dsfont, amssymb}
\usepackage{amsfonts,dsfont,amssymb,amsthm,stmaryrd,bbm}
\usepackage[cmex10]{amsmath} 
\usepackage{soul}

\usepackage{hyperref,mathtools}
\hypersetup{colorlinks=true,pdftitle=""}

\newtheorem{thm}{Theorem}[section]
\newtheorem{remark}[thm]{Remark}
\newtheorem{lem}[thm]{Lemma}
\newtheorem{prop}[thm]{Proposition}
\newtheorem{coro}[thm]{Corollary}

\newcommand\independent{\protect\mathpalette{\protect\independent}{\perp}} 
\def\independent#1#2{\mathrel{\rlap{$#1#2$}\mkern2mu{#1#2}}}

\newcommand{\supp}{\mathrm{Supp}}

\newcommand{\R}{\mathbb{R}}

\DeclareMathOperator{\Var}{Var}

\newcommand{\vol}{\mathrm{Vol}}

\newcommand{\Z}{\mathbb{Z}}

\newcommand{\Prob}{\mathbb{P}}

\def\Var{{\rm Var}}

\def\phi{\varphi}

\def\bee{\begin{eqnarray*}}
\def\ene{\end{eqnarray*}}
\usepackage{mathtools}

\title{A note on the Littlewood-Offord problem for discrete log-concave distributions}
\author{Arnaud Marsiglietti\footnote{Department of Mathematics, University of Florida, Gainesville, FL 32611, USA, E-mail: a.marsiglietti@ufl.edu} \, and James Melbourne\footnote{Probabilidad y Estad\'isticas, Centro de Investigaci\'ones en Matem\'aticas, Guanajuato, GTO 36023, MX, E-mail: james.melbourne@cimat.mx}}

\date{}

\begin{document}

\maketitle

\begin{abstract}

We present an extension of the famous Littlewood-Offord problem when Bernoulli distributions are replaced with discrete log-concave distributions. A variant of the Littlewood-Offord problem for arithmetic progressions, as well as an entropic version, is also discussed. Along the way, we recover and extend a result of Madiman and Woo (2015) on the entropy power inequality for discrete uniform distributions.

\end{abstract}


\section{Introduction}

Given $a = (a_1, \dots, a_n) \in (\R \setminus \{0\})^n$ and independent random variables $X_1, \dots, X_n$, $n \geq 1$, with a Rademacher distribution, that is, for all $1 \leq k \leq n$, $\Prob(X_k= \pm 1) = 1/2$,
the question of estimating the quantity
$$ \sup_{x \in \R} \, \Prob(a_1 X_1 + \cdots + a_n X_n = x) $$
is often referred to as the Littlewood-Offord problem. A classical result of Littlewood and Offord \cite{LO} and Erd\"os \cite{E} states that
\begin{eqnarray}\label{erdos}
\sup_{a \in (\R \setminus \{0\})^n} \sup_{x \in \R} \, \Prob(a_1 X_1 + \cdots + a_n X_n = x) \leq \frac{1}{2^n}\binom{n}{\lfloor \frac{n}{2} \rfloor} = O \bigg( \frac{1}{\sqrt{n}} \bigg).
\end{eqnarray}

Kleitman extended this result when $a_1, \dots, a_n$ are vectors in a Hilbert space \cite{K1}, \cite{K2}. Many variants of the Erd\"os-Littlewood-Offord problem have been established, such as improved bounds under certain constraints on the $a_i$'s \cite{H}, an inverse Littlewood-Offord theorem \cite{TV}, \cite{NV}, and a resilience version \cite{BFK}, but most are mainly dealing with Bernoulli distribution with parameter $1/2$. Recently, Fox, Kwan and Sauermann \cite[Question 6.2]{FKS} asked whether the Littlewood-Offord problem can be solved for Bernoulli distribution of arbitrary parameter $p \in (0,1)$. This question has been investigated by Singhal \cite{S}, who gave a qualitative solution to the problem, showing that a maximizer is obtained for $a_i \in \{-1, 1\}$ and gave Fourier theoretic bounds. Sharp quantitative bounds were found by Madiman, Melbourne, and Roberto in \cite{MMR} where an entropic generalization of the problem was considered. The case of general random variables was considered recently by Juškevičius and Kurauskas \cite{JK}.

Recall that an integer-valued random variable $X$ is said to be discrete log-concave if its probability mass function $p$ satisfies
$$ p(j)^{2} \geq p(j-1)p(j+1) $$
for all $j \in \mathbb{Z}$ and the support of $X$ is an integer interval. Discrete log-concave distributions form an important class. Examples include discrete uniform, Bernoulli, binomial
and convolutions of Bernoulli distributions with arbitrary parameters, Poisson, geometric, negative binomial, etc. (cf. \cite{JG} and references therein). We refer to \cite{Sta}, \cite{Bre}, \cite{SW}, \cite{Bra}  for further background on log-concavity.

The goal of this article is to extend the solution of the Littlewood-Offord problem to the whole class of discrete log-concave distributions. Our bounds are quantitative and non-asymptotic. In particular, we prove the following.

\begin{thm}\label{main1}

Let $X_1, \dots, X_n$, $n \geq 1$, be independent discrete log-concave random variables finitely supported. Then,
\begin{equation}\label{1}
\sup_{a \in (\R \setminus \{0\})^n} \sup_{x \in \R} \, \Prob(a_1 X_1 + \cdots + a_n X_n = x) \leq \frac{1}{\sqrt{1 + c \sum_{k=1}^n \Var(X_k)}},
\end{equation}
with $c=1$. Moreover, one may take $c = 2$ when the random variables are, in addition, symmetric about a point.

\end{thm}

We note that this recovers the $O(1/\sqrt{n})$ bound for independent Bernoulli distribution with parameter $1/2$.

It has been shown in \cite{MMR} that the bound \eqref{1} holds with $c=2$ when the $X_k$'s have a Bernoulli distribution with arbitrary parameters.

Theorem \ref{main1} is sharp up to an absolute constant, as the left-hand side of inequality \eqref{1} can be lower bounded by
$$ \frac{1}{\sqrt{1 + 12\sum_{k=1}^n \Var(X_k)}}, $$
see Remark \ref{sharp}.

The article also presents an entropic version of the Littlewood-Offord problem. See Section~\ref{2} for a precise definition of the R\'enyi entropy power $N_{\alpha}$.

\begin{thm}\label{main2}
Let $\alpha \in [0, +\infty]$ and $n \geq 1$. Let $X_1, \dots, X_n$ be independent discrete log-concave random variables finitely supported. Then,
\begin{equation}\label{entrop}
\inf_{a \in (\R \setminus \{0\})^n} N_\alpha(a_1 X_1 + \cdots + a_n X_n) \geq 1 + c \sum_{k=1}^n \Var(X_k),
\end{equation}
with $c=1$. Moreover, one may take $c = 4$ when $1 < \alpha \leq 2$, and for other values of $\alpha$ one may take $c=2$ when the random variables are, in addition, symmetric about a point.
\end{thm}

It has been shown in \cite{MMR} that the bound \eqref{entrop} holds with $c = \frac{2 \alpha}{\alpha - 1}$, $\alpha \geq 2$, when the $X_k$'s have a Bernoulli distribution. It turns out that Theorem \ref{main1} is a particular case of Theorem \ref{main2}.

We also present a version of the Littlewood-Offord problem for arithmetic progressions.

\begin{thm}\label{main3}
Let $X_1, \dots, X_n$, $n \geq 1$, be independent discrete log-concave random variables finitely supported. Then,
\begin{equation*}
\sup_{a \in (\R \setminus \{0\})^n} \sup_{x \in \R} \, \Prob(a_1 X_1 + \cdots + a_n X_n \in A_{l,m}(x)) \leq \frac{l}{\sqrt{1 + c \sum_{k=1}^n \Var(X_k) + c \, \frac{l^2-1}{12}}},
\end{equation*}
with $c=1$. Moreover, one may take $c = 2$ when the random variables are, in addition, symmetric about a point.
\end{thm}
Here, $A_{l,m}(x)$ is an arithmetic progression of length $l \geq 1$, that is $A_{l,m}(x) = \{x + mj \}_{j=1}^l$ for $m \in \mathbb{Z}$ and $x \in \mathbb{R}$. In fact, $m$ can be taken as a real number (see Section \ref{4}). For example, if the $X_k$'s are i.i.d. Bernoulli with parameter $1/2$, then we deduce
\[
    \sup_{a \in (\R \setminus \{0\})^n} \sup_{x \in \R} \, \Prob(a_1 X_1 + \cdots + a_n X_n \in A_{l,m}(x)) \leq \frac l {\sqrt{ 1 + \frac n 2 + \frac{l^2-1}{6}}}.
\]
We refer to Section \ref{4} for an estimate that improves Theorem \ref{main3} when the $X_k$'s have a Bernoulli distribution with arbitrary parameter $p \in (0,1)$, $p \neq \frac{1}{2}$.  

Let us note that the case $l = 1$ corresponds to the classical Littlewood-Offord problem, hence Theorem \ref{main1} is a particular case of Theorem \ref{main3}.

Finally, our method allows us to establish the following entropy power inequality for discrete uniform distributions.

\begin{thm}\label{EPI-unif}
    Let $\alpha \in [0, 2]$ and $n \geq 1$. Let $U_1, \dots, U_n$ be uniformly distributed independent random variables on any set of integers. Then,
    $$ N_{\alpha} \left( \sum_{k=1}^n U_k \right) \geq \sum_{k=1}^n N_{\alpha}(U_k) - (n-1). $$
\end{thm}

The particular case $\alpha = 1$ and $n=2$ of Theorem \ref{EPI-unif} has been established by Madiman and Woo in \cite{MW}.

The article is organized as follows. Section \ref{prel} provides the necessary background on the notion of majorization and rearrangement inequalities, which are the main ingredients in the proofs of Theorems \ref{main1} and \ref{main2}. The proofs of Theorems \ref{main1} and \ref{main2} are postponed to Section~\ref{2}. Section \ref{sec-4} demonstrates the advantage of our results compared with general bounds on the concentration function existing in the literature. Section \ref{4} presents a variant of the Littlewood-Offord problem for arithmetic progressions, in particular, Theorem \ref{main3} is proved. The last section focuses on the Bernoulli and uniform distributions, where Theorem \ref{EPI-unif} is proved.

\section{Preliminaries}\label{prel}

Throughout the article, we use the notation $x \cdot y = \sum_{i=1}^n x_i y_i$, for the dot product of $x=(x_1, \dots, x_n)$ and $y=(y_1, \dots, y_n)$. We also denote
\begin{equation}\label{M}
M(X) = \sup_{x \in \R} \, \Prob(X = x).
\end{equation}

We will need the following result.

\begin{thm}[\cite{A}, \cite{BMM}]\label{LC}

If the random variable $X$ has a discrete log-concave distribution, then
\begin{equation}\label{M-bound}
\frac{1}{\sqrt{1 + 12\Var(X)}} \leq M(X) \leq \frac{1}{\sqrt{1 + \Var(X)}}.
\end{equation}
Moreover, if the distribution of $X$ is symmetric about a point, then the upper bound may be sharpened to
$$ M(X) \leq \frac{1}{\sqrt{1 + 2\Var(X)}}. $$

\end{thm}

Let us note that the lower bound in \eqref{M-bound} holds for arbitrary random variables. The upper bound in \eqref{M-bound} was proven by Aravinda \cite{A}, who refined the bound
$$ M(X) \leq \frac{2}{\sqrt{1 + 4\Var(X)}} $$
obtained by Bobkov and the authors in \cite{BMM} (see also \cite{GMPP}, \cite{JMNS}).

Recall that a probability distribution $p = (p_1, \dots, p_n)$ written in descending order with positive entries, $p_i \geq p_{i+1}$, is majorized by another $q = (q_1, \dots, q_m)$ also written in descending order with positive entries when 
\begin{align}
    \sum_{i=1}^k q_i \geq \sum_{i=1}^k p_i
\end{align}
holds for all $k$. We write $p \prec q$ when $p$ is majorized by $q$. We also write $X \prec Y$ when the probability mass function of $X$ is majorized by the probability mass function of $Y$. After extending $q$ to $\{1, \dots, n\}$ by setting $q_j = 0$ for $m+1 \leq j \leq n$, this is equivalent to the existence of a doubly stochastic matrix $A$ such that $Aq = p$ (see \cite{MOA}). 

For example, if $q$ is a point mass, $(q_1, \dots, q_n) = (1,0,\dots,0)$, then for any other $p = (p_1, \dots, p_n)$ we can write
\begin{align}
    A = \left( \begin{array}{cccc}
    p_1 &p_2 &\dots &p_n \\
    p_2 &p_3 & \dots &p_{1} \\
    \vdots &\vdots & \vdots & \vdots \\
    p_n & p_1 & \dots & p_{n-1}
    \end{array} 
    \right),
\end{align}
so that $Aq = p$. More generally, if $q$ is not a probability sequence and $q = (M,0,\dots, 0)$ while $\sum_i p_i = M$ then setting $\lambda_i = p_i/M$, we can write
\begin{align} \label{eq: double stochastic matrix for point mass}
    A = \left( \begin{array}{cccc}
    \lambda_1 &\lambda_2 &\dots &\lambda_n \\
    \lambda_2 &\lambda_3 & \dots &\lambda_{1} \\
    \vdots &\vdots & \vdots & \vdots \\
    \lambda_n & \lambda_1 & \dots & \lambda_{n-1}
    \end{array} 
    \right),
\end{align}
so that $Aq = p$.

\begin{lem} \label{lem: majorization lemma}
If $Y$ is a random variable taking finitely many values, and $f$ is a deterministic function, then $Y \prec f(Y)$.
\end{lem}

\begin{proof}
Without loss of generality, we may assume that $Y$ is distributed on $\{1,2,\dots, n\}$, and denote by $\{a_1, \dots, a_m\}$ the support of $f(Y)$. If the distribution of $Y$ is denoted by $p$, note that the distribution of $f(Y)$, written as $q$, will satisfy $q_i = \sum_{j \in f^{-1}(\{a_i\})} p_j$.  Writing $q = (q_1,0, \dots,0,q_2, 0, \dots,0, \dots, q_m, 0,\dots, 0)$ where the number of zeros between $q_i$ and $q_{i+1}$ is determined by the cardinality $n_i$ of $f^{-1}(\{a_i\})$, and writing $p$ in order such that $f^{-1}(\{a_i\}) = \{k_i, k_i+1, \dots, k_i+n_i-1\}$ so that $q_i = p_{k_i} + \cdots + p_{k_i + n_i - 1}$, based on the discussion above we can write a doubly stochastic block matrix,
\begin{align}
    \Lambda = \left( \begin{array}{ccccc}
        (A_1)&(0) &(0) &\dots & (0) \\
        (0) & (A_2) &(0) &\dots & (0) \\
        \vdots & \vdots & \ddots && \vdots \\
        \vdots & \vdots & & \ddots & \vdots \\
        (0) & (0) & \dots & (0) & (A_m)
        \end{array}
        \right)
\end{align}
where each $A_i$ is of the form of $\eqref{eq: double stochastic matrix for point mass}$ for the $p_j$ such that $f(j) = a_i$, so that $\Lambda q  = p$, and the lemma holds.
\end{proof}

If $f \colon \Z \to [0, +\infty)$ is finitely supported, with support $\{x_0, \dots, x_m\}$, denote by $f^\#$ its squeezed rearrangement, that is, $f^\#$ is supported on $\{0, \dots,m\}$ and $f^\#(j) = f(x_j)$, for $j\in\{0, \dots,m\}$. If $X$ is an integer-valued random variable with probability mass function $f$, we denote by $X^\#$ the random variable with probability mass function $f^\#$. The following result was proven in \cite{MWW}.

\begin{thm}[\cite{MWW}]\label{rearrance-2}
If $X_1, \dots, X_n$ are integer-valued independent random variables such that $X_1^\#, \dots, X_n^\#$ are log-concave, then
$$ X_1 + \dots + X_n \prec X_1^\# + \dots + X_n^{\#}. $$    
\end{thm}

Finally, recall that the R\'enyi entropy of order $\alpha \in (0,1) \cup (1, +\infty)$ of a discrete random variable $X$ with values in a countable set $I$ and with probability mass function $p$ is defined as
$$ H_{\alpha}(X) = \frac{1}{1-\alpha} \log \left( \sum_{x \in I} p^{\alpha}(x) \right) = \log(\|p\|_{\alpha}^{\frac{\alpha}{1-\alpha}}). $$
The limit cases are interpreted as
$$ H_0(X) = \log(|\supp(p)|), \quad H_1(X) = -\sum_{x \in I} p(x) \log(p(x)), \quad H_{\infty}(X) = -\log(\sup_{x \in I} p(x)). $$
We note that the $M$-functional defined in \eqref{M} may be viewed as a member of the family of R\'enyi entropies via the formula
\begin{equation}\label{member}
M(X) = e^{-H_{\infty}(X)}.
\end{equation}
In particular, considering the R\'enyi entropy power $N_{\alpha}(X) = e^{2H_{\alpha}(X)}$, Theorem \ref{LC} yields the bound
\begin{equation}\label{entropy}
N_{\alpha}(X) \geq N_{\infty}(X) \geq 1 + \Var(X),
\end{equation}
for arbitrary discrete log-concave random variable $X$, where the first inequality holds by monotonicity of R\'enyi entropy.

\section{The Littlewood-Offord problem for discrete log-concave distributions and an entropic version}\label{2}

Throughout this section, given random variables $X_1, \dots, X_n$, we denote $X=(X_1, \dots,X_n)$. The first step in establishing the Littlewood-Offord problem for log-concave distributions is to reduce the problem to signs.

\begin{thm}\label{thm: majorization by signs}

For $a_i \in \mathbb{R} \setminus \{0\}$, and $X_i$ independent, log-concave, $\mathbb{Z}$-valued random variables taking finitely many values, there exist signs $v_i \in \pm 1$ such that $a \cdot X \prec v \cdot X$.

\end{thm}

\begin{proof}
Let us first observe that $a \cdot X$ can be majorized by $\tilde{a} \cdot X$, for set of constants $\tilde{a}_i \in \mathbb{Z} \setminus \{0\}$. Indeed, it has been observed in \cite[Proof of Lemma 5.1]{MMR} that one may construct a linear map $T \colon \R \to \mathbb{Q}$ such that $T(a_i) \in \mathbb{Z} \setminus \{0\}$ for all $i$. Thus, by Lemma \ref{lem: majorization lemma}, $a \cdot X \prec T( a\cdot X) $. Further, since the $X_i$'s are integer-valued, one has
\begin{align}
T(a_1 X_1 + \cdots + a_n X_n)
    = T(a_1) X_1 + \cdots + T(a_n) X_n.
\end{align}
Writing $T(a) = (T(a_1), \dots, T(a_n)) \in (\mathbb{Z} \setminus \{0\} )^n$, we thus have $a \cdot X \prec T(a) \cdot X$.

Observe that our result follows from Theorem \ref{rearrance-2} since $T(a) \cdot X \prec (T(a_1) X_1)^{\#} + \cdots + (T(a_n) X_n )^{\#}$. Indeed, $(T(a_i) X_i)^{\#} = v_i X_i$ where $v_i = sign(T(a_i))$. Since $v_i X_i$ is log-concave one may apply Theorem \ref{rearrance-2} and we have $a \cdot X \prec T(a) \cdot X \prec v \cdot X$ where $v_i = \pm 1$ and the result follows.
\end{proof}

Since $\alpha$-R\'enyi entropy is Schur concave as a consequence of \cite[Proposition 3-C.1]{MOA} (see also \cite{MWW}),
the proof of Theorem \ref{main2} follows immediately.

\begin{proof}[Proof of Theorem \ref{main2}]
Theorem \ref{thm: majorization by signs} combined with Schur concavity of R\'enyi entropy yields that for $\alpha \in [0,\infty]$, $X = (X_1, \dots, X_n)$ with $X_i$ independent and log-concave  then for all $a = (a_1, \dots, a_n)$ with $a_i \in \mathbb{R} \setminus \{0\}$, there exists $v_i = \pm 1$ such that $v = (v_1, \dots, v_n)$ implies
\begin{align}
N_\alpha( a \cdot X) \geq N_\alpha( v \cdot X).
\end{align}
Moreover the choice of signs is independent of the $X_i$ determined only by the coefficients $a_i$.
Therefore, using \eqref{entropy},
$$ N_\alpha( a \cdot X) \geq N_\alpha( v \cdot X) \geq 1 +\Var(v \cdot X) = 1 + \sum_{i=1}^n \Var(X_i). $$
It has been shown in \cite{BMM} that when $1 < \alpha \leq 2$, $N_{\alpha}(X) \geq 1 + 4\Var(X)$ for arbitrary log-concave $X$, while $N_{\alpha}(X) \geq 1 + 2\Var(X)$ when $X$ is symmetric about a point and log-concave. This concludes the proof.    
\end{proof}

\begin{remark}\label{sharp}

Theorem \ref{main2} is sharp up to an absolute constant. This is a consequence of the bound $N_{\alpha}(X) \leq 1 + \frac{4(3 \alpha - 1)}{\alpha - 1} \Var(X)$ proved in \cite{BMM}, which holds for all $\alpha > 1$. Therefore,
$$ \inf_{a \in (\R \setminus \{0\})^n} N_\alpha(a \cdot X) \leq 1 + \frac{4(3 \alpha - 1)}{\alpha - 1} \sum_{i=1}^n \Var(X_i). $$
To obtain a better estimate when $\alpha$ tends to 1, one may use the following well-known upper bound for the discrete entropy, $N(X) \leq \frac{2 \pi e}{12} + 2\pi e \Var(X)$ (see, e.g., \cite{Massey}, \cite{BM}), which yields for all $\alpha \geq 1$,
$$ \inf_{a \in (\R \setminus \{0\})^n} N_\alpha(a \cdot X) \leq \frac{2 \pi e}{12} + 2 \pi e\sum_{i=1}^n \Var(X_i). $$
    
\end{remark}

Specializing Theorem \ref{main2} to $\alpha = +\infty$, and recalling \eqref{member}, we obtain Theorem \ref{main1}:
\begin{equation}\label{M-function}
\sup_{a \in (\R \setminus \{0\})^n} M(a \cdot X) \leq \frac{1}{\sqrt{1 + \sum_{i=1}^n \Var(X_i)}},
\end{equation}
holding for arbitrary independent log-concave random variables $X_1, \dots, X_n$. In particular, we deduce the following.

\begin{prop}\label{answer}

Let $X_1, \dots, X_n$ be i.i.d. Bernoulli distribution of parameter $p$. Then,
$$ \sup_{a \in (\R \setminus \{0\})^n} M(a \cdot X) \leq \frac{1}{\sqrt{1 + np(1-p)}}. $$

\end{prop}

As mentioned in the introduction, in this specific case of Bernoulli distribution, a refined argument has been used in \cite{MMR} to show that 
\begin{equation}\label{answer-2}
\sup_{a \in (\R \setminus \{0\})^n} M(a \cdot X) \leq \frac{1}{\sqrt{1 + 2np(1-p)}}.
\end{equation}

We note that one may provide a unification of both Erdos' result of the Littlewood-Offord problem \eqref{erdos} and \eqref{answer-2}.

\begin{prop}\label{extend}

Let $X_1, \dots, X_n$ be independent random variables such that for each $i \in \{1, \dots, n\}$, $X_i \in \{x_i, x_{i+1}\}$, where $x_i, x_{i+1} \in \Z$ with $x_i \leq x_{i+1}$ and 
$$ \Prob(X_i = x_i) = 1 - \Prob(X_i = x_{i+1}) = 1-\theta_i, \qquad \theta_i \in (0,1). $$
Then,
$$ \sup_{a \in (\R \setminus \{0\})^n} M(a \cdot X) \leq \frac{1}{\sqrt{1 + 2\sum_{i=1}^n \frac{\Var(X_i)}{(x_i-x_{i+1})^2}}}. $$

\end{prop}

\begin{proof}
Using the same argument as in the proof of Theorem \ref{thm: majorization by signs}, one may deduce that for all $a \in (\R \setminus \{0\})^n$,
$$ M(a \cdot X) \leq M(\mathbbm{1} \cdot B), $$
where $B_1, \dots, B_n$ are independent Bernoulli distributions of parameter $\theta_i$ or $1-\theta_i$ depending on the sign of $a_i$.
The result follows by using the bound
$$ M(\mathbbm{1} \cdot B) \leq \frac{1}{\sqrt{1 + 2\sum_{i=1}^n \Var(B_i)}}, $$
which is a consequence of \eqref{answer-2}, and by noting that
$$ \Var(B_i) = \theta_i(1-\theta_i) = \frac{\Var(X_i)}{(x_{i+1} - x_i)^2}. $$
\end{proof}

The Littlewood-Offord solution \eqref{erdos}, as well as Proposition \ref{answer} immediately follow from Proposition \ref{extend}. Note that one may even allow Rademacher $\pm 1$ distributions with arbitrary parameter $p \in (0,1)$, which yields the same inequality as for Bernoulli distributions. We state this result in the next corollary.

\begin{coro}

If the $X_i$'s are independent Rademacher distributions with arbitrary parameter $p \in (0,1)$, then
$$ \sup_{a \in (\R \setminus \{0\})^n} M(a \cdot X) \leq \frac{1}{\sqrt{1 + 2np(1-p)}}. $$

\end{coro}

\section{Comparison with general bounds on the concentration function}\label{sec-4}

The goal of this section is to demonstrate that our upper bound on
$$ M(X) = \sup_{x \in \mathbb{R}} \mathbb{P}(X = x) $$
given in equation \eqref{M-function}, specialized to $a=(1, \dots, 1)$, provides better information compared with existing results in the literature on concentration function. Let us recall that the concentration function of a random variable $X$ is defined as
$$ Q(X ; \lambda) = \sup_x \Prob(x \leq X \leq x + \lambda), \qquad \lambda \geq 0. $$
Note that for $X$ integer valued and $\lambda < 1$,
\[
    Q(X; \lambda) = M(X).
\]
A general bound was established by Miroshnikov and Rogozin \cite{M-R}.

\begin{thm}[Miroshnikov-Rogozin \cite{M-R}]
    There exists a universal constant $C > 0$ such that $2 \lambda \geq \lambda_k$ and $X_k$ independent random variables with $S = X_1 + \cdots + X_n$,
    \[
        Q(S; \lambda) \leq C \lambda \left( \sum_{k=1}^n \mathbb{E}\left( |X_k^s| \wedge \frac{\lambda_k}{2} \right)^2 Q^{-2}(X_k, \lambda_k) \right)^{-\frac 1 2},
    \]
    where given $X$, $X^s = X - X'$ and $X'$ is an independent copy of $X$.
\end{thm}

When $X_k \sim$ Bernoulli($p_k$), note that $|X_k^s| \sim$ Bernoulli$(2p_k(1-p_k))$ so that for $\lambda < 1$,
\[
    M(S) = Q(S;\lambda) \leq C \lambda \left( \sum_{k=1}^n  \frac{\lambda_k^2 \,2 p_k(1-p_k) }{4}  M^{-2}(X_k) \right)^{-\frac 1 2}.
\]
Minimizing the right hand side with $\lambda_k = 2 \lambda$ we have
\begin{equation}\label{MR}
    M(S) \leq C \left( \sum_{k=1}^n \frac{2 p_k (1-p_k)}{(p_k \vee (1-p_k))^2}\right)^{- \frac 1 2} = C \left( \sum_{k=1}^n \frac{ 2 \Var(X_k)}{(p_k \vee (1-p_k))^2} \right)^{- \frac 1 2}.
\end{equation}
Since $\frac 1 2 \leq p \vee (1-p) \leq 1$ for $p \in [0,1]$, the right-hand side of \eqref{MR} is at least $\frac{C}{\sqrt{2}} \left( \sum_{k=1}^n \Var(X_k) \right)^{-\frac 1 2}$. Therefore for Bernoulli distributions, even after optimizing, the Miroshnikov-Rogozin yields the bound
$$ M(S) \leq \frac{\widetilde{C}}{\sqrt{\Var(S)}} $$
for some absolute constant $\widetilde{C} > 0$. 

Note that in the absence of an explicit constant $\widetilde{C}$, this inequality is only interesting in the regime that $\Var(S) \to \infty$, where central limit theorems are often more viable.  In particular, it yields trivial results in the Poisson regime, namely $M(S) \leq O(1)$, when the variance of $S$ can be bounded away from zero. In regimes where the variance tends to $0$, the result is of an order worse than the trivial $M(S) \leq 1$. Whereas, our bounds derived in Section \ref{2} provide meaningful quantitative estimates in all regimes.

In more recent work, explicit constants have been obtained for a variant of the Miroshnikov-Rogozin inequality. To this end, for $\theta$ a unit vector in $\mathbb{R}^n$ define
\[
    p_\theta(t) = \vol_{n-1} \left\{ x \in \R^n : \|x\|_\infty \leq \frac 1 2, \langle x, \theta \rangle = t \right\}.
\]

\begin{thm}[Bobkov-Chistyakov \cite{BC}]
    Given $\lambda \geq \left(\sum_{k=1}^n \lambda_k^2 \right)^{\frac 1 2}$ and $X_k$ independent random variables with $S = X_1 + \cdots + X_n$,
    \[
        Q(S; \lambda) \leq  \frac{ 2^{\frac 3 2} \lambda}{c}   \left( \sum_{k=1}^n \lambda_k^2 Q^{-2}(X_k; \lambda_k) \right)^{- \frac 1 2},
    \]
    where  $$c \coloneqq  \inf_{|t| <\frac 1 2, \|\theta\| = 1 } p_\theta(t).$$ 
\end{thm}
Moreover, Bobkov and Chistyakov showed that $c \geq 0.00095$ independent of dimension.
Later Melbourne, Tkocz, and Wyczesany \cite{MTW} showed that the body $\{ x \in \R^n : \|x\|_\infty \leq \frac 1 2 \}$ can be replaced by any isotropic convex body $K$.  More precisely for a convex body $K$ such that  $\int_K x_i x_j dx = L_K^2 \delta_{ij}$,  for some constant $L_K > 0$, 
\[
    \inf_{\theta, |t| < L_K \sqrt{3}}  \vol_{n-1} \left \{ x \in K : \langle x , \theta \rangle = t \right \} \geq \frac{e^{- \sqrt{6}}}{\sqrt{2 L_K^2}}.
\]
This result is proven sharp for high dimensional double cones, but even in the case of the cube where $L_K = \frac{1}{\sqrt{12}}$ it gives $c \geq \sqrt{6} e^{- \sqrt{6}}$, so that Bobkov and Chistyakov can be stated as 

    \[
        Q(S; \lambda) \leq  \frac{2 e^{\sqrt{6}}}{\sqrt{3}} \lambda  \left( \sum_{k=1}^n \lambda_k^2 Q^{-2}(X_k; \lambda_k) \right)^{- \frac 1 2}.
    \]
For reference, numerically $\frac{2 e^{\sqrt{6}}}{\sqrt{3}} \ \approx 13.3742$. Hence for $C =\frac{2 e^{\sqrt{6}}}{\sqrt{3}} $, $X_k$ integer valued and $\sum_{k=1}^n \lambda_k^2 \leq \lambda^2 < 1$, by Bobkov-Chistyakov
\[
    M(S) = Q(S; \lambda) \leq C \lambda \left( \sum_{k=1}^n \lambda_k^2 Q^{-2}(X_k; \lambda_k ) \right)^{- \frac 1 2} = C \lambda \left( \sum_{k=1}^n \lambda_k^2 M^{-2}(X_k )\right)^{- \frac 1 2}.
\]
Minimizing the right-hand side over $\lambda_k$, that is when $\lambda_j = \lambda$ for $j$ such that $M^{-2}(X_j) = \max_{1 \leq k \leq n} M^{-2}(X_k)$ and $\lambda_l = 0$ for $l \neq j$, gives
\[
    M(S) \leq C \left( \max_{1 \leq k \leq n} M^{-2}(X_k) \right)^{- \frac 1 2} = C \min_{1 \leq k \leq n} M(X_k).
\]
However, this is a trivial bound as Young's convolution inequality yields 
$$ M(S) \leq \min_{1 \leq k \leq n} M(X_k). $$

\section{A Littlewood-Offord type problem for arithmetic progressions}\label{4}

Given independent $\mathbb{Z}$-valued random variables $X_1, \dots, X_n$, $n \geq 1$, and $a = (a_1, \dots, a_n) \in (\mathbb{R} \setminus \{0\} )^n$, we ask for an upper bound on 
\[
    \sup_{x \in \R} \, \mathbb{P}(a \cdot X \in A_{l,m}(x))
\]
where $A_{l,m}(x)$ is an arithmetic progression of length $l \geq 1$, that is $A_{l,m}(x) = \{x + mj \}_{j=1}^l$ for $m, x \in \mathbb{R}$.  In the case that $l = 1$ this corresponds to the classical Littlewood-Offord problem for the variables $X_k$.

\begin{proof}[Proof of Theorem \ref{main3}]
Let $Y$ be a discrete random variable independent of $U_l$, where $U_l$ is uniform on $\{1,2,\dots, l\}$. Then,
\begin{align*}
    \mathbb{P}(Y - m U_l = x) 
        &= 
            \sum_{k=1}^l \mathbb{P}(U_l = k, Y = x + mk ) 
                \\
        &=
            \frac 1 l \sum_{k=1}^l \mathbb{P}(Y = x +mk)
                \\
        &=
            \frac 1 l \mathbb{P}( Y \in A_{l,m}(x)).
\end{align*}
Thus, this Littlewood-Offord problem for arithmetic progressions is equivalent to finding upper bounds on $M(a \cdot X - m U_l)$. When the $X_k$'s are discrete log-concave, one may thus apply Theorem \ref{main1} to obtain
$$ \sup_{x \in \R} \, \mathbb{P}(a \cdot X \in A_{l,m}(x)) = l \, M(a \cdot X - m U_l) \leq \frac{l}{\sqrt{1 + c \left( \sum_{k=1}^n \Var(X_i) + \Var(U_l) \right)}}, $$
which is the desired result since $\Var(U_l) = (l^2-1)/12$.
\end{proof}

\begin{remark}
Let us note that Theorem \ref{main3} is sharp up to an absolute constant when taking supremum over all $m \neq 0$, as
\begin{equation*}
\sup_{m \neq 0} \sup_{a \in (\R \setminus \{0\})^n} \sup_{x \in \R} \, \Prob(a_1 X_1 + \cdots + a_n X_n \in A_{l,m}(x)) \geq \frac{l}{\sqrt{1 + 12 \sum_{k=1}^n \Var(X_k) + l^2-1}},
\end{equation*}
due to the lower bound in \eqref{M-bound}.
\end{remark}

\section{Specific case of the Bernoulli and uniform distributions}

\subsection{Bernoulli distribution}

This section focuses on strengthening Theorem \ref{main3} for the Bernoulli distribution. Let $U_l$ be uniform on $\{1, \dots, l\}$. It has been shown in \cite[Proof of Theorem 1.7]{MR} that for all $p \geq 2$ the Fourier transform of $U_l$ satisfies
\[
    \int_{- \frac 1 2}^{\frac 1 2} \left| \mathbb{E} e^{2 \pi i t U_l} \right|^p dt \leq \int_{- A}^{A} e^{- \pi (l^2 - 1) p t^2 /2} dt,
\]
where $A$ is determined implicitly through the equation
\begin{equation}\label{A-def}
    \int_{-A}^A e^{- \pi (l^2 - 1) t^2} dt = \int_{-\frac 1 2 }^{\frac 12} \left| \mathbb{E}e^{2 \pi i t U_l} \right|^2 dt = \frac 1 l.
\end{equation}
The first equality in \eqref{A-def} gives the implicit definition of $A$ (in terms of $l$), while the second equality comes from recalling that $U_l$ is uniform and using Parseval identity. Therefore,
\begin{equation}\label{A-1} \frac 1 {2\pi} \int_{-\pi}^{\pi} \left| \mathbb{E} e^{itU_l} \right|^p dt \leq 2  \int_{0}^{A} e^{- \pi (l^2 - 1)p t^2 /2} dt = \frac{2}{\sqrt{ \pi  (l^2-1) p}} \int_0^{\sqrt{\pi (l^2-1)A^2 p}} e^{-x^2/2} dx = 2A \Phi (c p),
\end{equation}
where $\Phi(x) \coloneqq \frac 1 {\sqrt{x}} \int_0^{\sqrt{x}} e^{-t^2/2} dt$ and $c = \pi (l^2 -1) A^2$. Note that 
\[
A = \frac{1}{\sqrt{ \pi(l^2 - 1) }} \cdot \operatorname{erf}^{-1}\left( \frac{\sqrt{l^2 - 1}}{l} \right),
\]
where $\operatorname{erf}(x) = \frac{2}{\sqrt{\pi}} \int_0^{x} e^{-t^2} dt$ is the error function. On the other hand, it has been shown in \cite[Theorem 2.8]{MMR} that if $X$ is a Bernoulli random variable with variance $\sigma^2$ and $q \geq 1$,
\begin{equation}\label{A-2}
        \frac 1 {2\pi} \int_{-\pi}^{\pi} \left| \mathbb{E} e^{itX} \right|^q dt \leq \frac 1 {\sqrt{6 \sigma^2 q}} \int_0^{\sqrt{6 \sigma^2 q}} e^{-t^2/2} = \Phi(6 \sigma^2 q).
\end{equation}
    
Let us now consider $X_1, \dots, X_n$ independent Bernoulli random variables with variance $\sigma_k^2$ and $U_l$ uniform on $\{1, \dots, l\}$, and denote $v=(\pm1, \dots, \pm 1)$ and $v_0=\pm 1$ any choice of signs, then
\begin{eqnarray*}
\|f_{\sum_{k=1}^n v_k X_k + v_0 U_l}\|_{\infty} & \leq & \|\widehat{f}_{\sum_{k=1}^n v_k X_k + v_0 U_l}\|_{1} \\ & = & \frac 1 { 2\pi } \int_{- \pi}^{\pi} \left| \mathbb{E} e^{i t (v \cdot X  + v_0 U_l)} \right| dt \\ & \leq & \frac{1}{2 \pi} \left(\int_{- \pi}^{\pi} \left|\mathbb{E} e^{it U_l}\right|^p dt \right)^{\frac 1 p} \prod_{k=1}^n \left(\int_{-\pi}^{\pi} \left| \mathbb{E} e^{it X_k} \right|^{q_k} dt \right)^{\frac 1 {q_k}} \\ & \leq &  \left( 2A \Phi (\pi (l^2 -1) A^2 p) \right)^{\frac{1}{p}} \, \prod_{k=1}^n \Phi \left( 6 \sigma_k^2 q_k \right)^{\frac{1}{q_k}},
\end{eqnarray*}
where we have used the Hausdorff-Young inequality, H\"older's inequality with $\frac 1 p + \sum_{k=1}^n \frac 1 {q_k} = 1$,  the independence of the variables, and the bounds \eqref{A-1}, \eqref{A-2}. Choosing $C = \pi (l^2-1)A^2 + 6 \sum_{k=1}^n \sigma_k^2$, and setting $q_k = \frac{C}{6 \sigma_k^2}  $ and $p = \frac{C}{\pi (l^2-1) A^2}$, then if $p \geq 2$ we have
\begin{align*}
        \sup_{x \in \R} \mathbb{P}(a \cdot X \in A_l)
            &\leq 
                l \, \|f_{\sum v_i X_i + v_0 U_l}\|_{\infty}
                    \\
            & \leq l (2A)^{\frac{1}{p}} \Phi (C) ^{\frac{1}{p}} \, \prod_{k=1}^n \Phi \left( C \right)^{\frac{1}{q_k}}\\
            &=
                l (2A)^{\frac{1}{p}} \Phi(C).
\end{align*}
Using the bound $\Phi(z) \leq \frac{1}{\sqrt{1+\frac{z}{3}}}$ holding for all $z > 0$ (see \cite[Lemma 2.9]{MMR})
we deduce
\begin{equation}\label{Bern}
\sup_{x \in \R} \mathbb{P}(a \cdot X \in A_l) \leq  (2A)^{\frac{1}{p}} \frac{l}{\sqrt{1 + 2 \sum_{k=1}^n \Var(X_k) + \frac{l^2-1}{12} 4 \pi A^2} }.
\end{equation}

\begin{remark}
    
Let us compare the bound \eqref{Bern} with theorem \ref{main3}. Since $2A \leq 1$, we have 
$$ (2A)^{\frac{1}{p}} \frac{l}{\sqrt{1 + 2 \sum_{k=1}^n \Var(X_k) + \frac{l^2-1}{12} 4 \pi A^2} } \leq \frac{l}{\sqrt{1 + 2 \sum_{k=1}^n \Var(X_k) + \frac{l^2-1}{12} 4 \pi A^2} }. $$

Therefore, for $l=1$, we recover the bound \eqref{answer-2} proved in \cite{MMR}. Moreover, for $l = 2$, one can check numerically that $4 \pi A^2 \geq 1$, so that the bound \eqref{Bern} is always better for Bernoulli distribution than Theorem \ref{main3}. Note that for fixed length $l$, the bound \eqref{Bern} is stronger as the variance grows.

However, the bound is not always applicable when $l \geq 2$, as $p = 1 + \frac{6 \sum_{k=1}^n \Var(X_k)}{\pi (l^2-1) A^2}$ needs to be greater than or equal to 2. Hence, one may not choose variances that are too small compared to the length $l$.

\end{remark}

\subsection{Discrete uniform distribution}

This section specializes to the uniform distribution. In particular, we recover and extend a result of Madiman and Woo \cite{MW} on the entropy power inequality for discrete uniform distributions.

Let $U$ be a uniform distribution on consecutive integers $\{a, \dots, b\}$, with $a,b \in \Z$. Denote $l = b-a+1$. Note that $\Var(U) = (l^2-1)/12$, therefore its Fourier transform satisfies
$$ \|\widehat{f_U}\|_2^2 = \frac{1}{l} = \frac{1}{\sqrt{1 + 12 \Var(U)}}, $$
where the first identity follows from Parseval. Therefore, using \cite[Lemma 2.7]{MMR}, we obtain for $n$ independent uniformly distributed random variables $U_1, \dots, U_n$
$$ \|\widehat{f}_{\sum_{k=1}^n U_k}\|_2^2 \leq \frac{1}{\sqrt{1 + 12 \sum_{k=1}^n \Var(U_k)}}. $$
Using the Parseval identity, this leads to
\begin{equation*}
H_2 \left( \sum_{k=1}^n U_k \right) = \log(\|p_{\sum_{k=1}^n U_k}\|_2^{-2}) = \log(\|\widehat{f}_{\sum_{k=1}^n U_k} \|_2^{-2}) \geq \frac{1}{2} \log \left( 1 + 12 \sum_{k=1}^n \Var(U_k) \right).
\end{equation*}
By monotonicity of entropy, we deduce that for all $\alpha \leq 2$,
\begin{equation}\label{unif}
H_{\alpha} \left( \sum_{k=1}^n U_k \right) \geq \frac{1}{2} \log \left( 1 + 12 \sum_{k=1}^n \Var(U_k) \right).
\end{equation}
We are now ready to prove Theorem \ref{EPI-unif}.

\begin{proof}[Proof of Theorem \ref{EPI-unif}]
    Let $\alpha \leq 2$. Denoting $\Delta_\alpha(X) = N_\alpha(X) - 1$, which reflects the variance better than the entropy power for discrete distributions, we deduce from \eqref{unif} that
\[
    \Delta_\alpha \left( \sum_{k=1}^n U_k \right) \geq 12 \sum_{k=1}^n \Var(U_k).
\]
However, we note that for a uniform random variable $U$ on an integer interval, $\Delta_\alpha(U) = 12 \ \Var(U)$, for any $\alpha$. Thus we have
\begin{equation}\label{bound-unif}
    \Delta_\alpha \left( \sum_{k=1}^n U_k \right) \geq \sum_{k=1}^n \Delta_\alpha(U_k),
\end{equation}
for $\alpha \leq 2$. Moreover, for $X_k$ uniform on any set of integers, $X_k^{\#}$ has a uniform distribution on an integer interval and hence is log-concave, thus $X_k^{\#} \sim U_k$ for $U_k$ uniform on an integer interval, therefore
\begin{equation*}
\Delta_\alpha \left( \sum_{k=1}^n X_k \right) \geq \Delta_\alpha \left( \sum_{k=1}^n X_k^{\#} \right) \geq \sum_{k=1}^n \Delta_\alpha \left( X_k^{\#} \right) = \sum_{k=1}^n \Delta_\alpha(X_k),
\end{equation*}
where the first inequality comes from Theorem \ref{rearrance-2} together with Schur concavity of R\'enyi entropy, and the second inequality from \eqref{bound-unif}.
\end{proof}

\begin{remark}
    \begin{enumerate}
        \item Let us note that \eqref{bound-unif} implies Theorem \ref{main2} with $c=12$ and $\alpha \in [0,2]$ when the random variables are uniformly distributed.

        \item Employing the relation
$$ H_{\alpha}(X) \leq H_{\infty}(X) + \log(\alpha^{\frac{1}{\alpha - 1}}) $$
obtained in \cite{MT}, which is valid for all log-concave distributions and $0 < \alpha < \infty$, we deduce by taking $\alpha =2$ that
$$ N_{\infty} \left( \sum_{k=1}^n U_k \right) \geq \frac{1}{4} N_{2} \left( \sum_{k=1}^n U_k \right) \geq \frac{1}{4} + 3  \sum_{k=1}^n \Var(U_k), $$
where the second inequality comes from \eqref{unif}. Equivalently,
$$ M \left( \sum_{k=1}^n U_k \right) \leq \frac{1}{\sqrt{\frac{1}{4} + 3  \sum_{k=1}^n \Var(U_k)}}, $$
which is an improvement of Theorem \ref{main1} whenever the random variables are uniformly distributed on at least 3 points.
    \end{enumerate}
\end{remark}

\end{document}